\documentclass[11pt]{amsart}

\usepackage{amsmath, amssymb, amsthm}
\usepackage[margin=1in]{geometry}

\newtheorem{theorem}{Theorem}
\newtheorem{lemma}[theorem]{Lemma}
\newcommand{\eps}{\varepsilon}

\title{Sign changes of the Liouville function in arithmetic progressions}

\author{Kevin Ford}
\author{Maksym Radziwi\l\l}

\dedicatory{To Roger Heath-Brown on the occasion of his 75th birthday}
\begin{document}


\maketitle

\section{Introduction}
 
Dirichlet established that for any reduced residue class $(a,q) = 1$ there exists a prime $p$ with $p \equiv a \pmod{q}$. It is conjectured, based on probabilistic heuristics, that the least such prime $p$ satisfies $p \ll_\varepsilon q^{1 + \varepsilon}$ for any $\varepsilon > 0$;
see \cite{LPS} for a more precise conjecture.
 Towards this conjecture,
Chowla~\cite{Chowla} showed in 1934 that the generalized Riemann
hypothesis implies $p \ll_\varepsilon q^{2+\varepsilon}$ for any $\varepsilon > 0$. This corresponds to the natural
square root barrier for equidistribution problems in arithmetic
progressions. A decade later, Linnik~\cite{Linnik1, Linnik2} proved
unconditionally that $p \ll q^{L}$ for some absolute constant~$L$.
Considerable effort has since gone into reducing the admissible value
of~$L$. A major advance was made by Heath-Brown~\cite{HeathBrown}, who
proved $L = 5.5$. Building on Heath-Brown's work,
Xylouris~\cite{Xylouris} refined the argument to show that $L = 5$
is admissible. This remains the best unconditional result to date.

In this short note we consider a multiplicative function analogue. For a prime $q$ and $(a,q) = 1$, we ask: how large must $N$ be to guarantee the existence of integers $n, m \leq N$ with $n \equiv m \equiv a \pmod{q}$ and $\lambda(n) = -1$, $\lambda(m) = +1$, where $\lambda$ denotes the Liouville function?
 
\begin{theorem}\label{thm:main}
    Let $\varepsilon > 0$. Then, for any prime $q$, sufficiently large with respect to $1 / \varepsilon$ and any residue class $(a,q) = 1$, there exist integers $n,m \leq q^{5/2 + \varepsilon}$ with $n \equiv m \equiv a \pmod{q}$ such that $\lambda(n) = -1$ and $\lambda(m) = 1$.  
\end{theorem}
 
The present result was obtained by the authors in 2018, and a complete proof was shared with Matom\"aki  in July 2023, at the conference honoring Heath-Brown's 70th birthday, in response to her work with Ter\"av\"ainen \cite{MT_products} establishing the existence of products of three primes $p_1 p_2 p_3 \ll q^3$ in every progression $a \pmod{q}$ with $(a,q) = 1$. In upcoming work \cite{MT}, using methods different from ours, Matom\"aki and Ter\"av\"ainen have obtained the stronger bound $q^{2+\varepsilon}$ for general real-valued multiplicative functions unless the sign of the function strongly pretends to be a real Dirichlet character modulo $q$.  For the Liouville function, their result gives $q^{2+\varepsilon}$ unconditionally, matching what was previously known only under the assumption of the Generalized Riemann Hypothesis. 
 
Our argument is short and proceeds along different lines from \cite{MT}. We show that if $\lambda(n)$ takes a constant sign for all $n \equiv a \pmod{q}$ with $n \leq Xq^2$ then $\lambda(n)$ is $q$-periodic for all $n \leq Xq$ with $(n, q) = 1$ and the sum $\lambda(1) + \lambda(2) + \ldots + \lambda(q - 1)$ is exactly zero. We then show, using the large sieve and Poisson summation, that for $X > q^{1/2 + 3\varepsilon}$ this implies that $\lambda(n) = \kappa' \chi_q(n)$ with $\kappa' \in \{-1,1\}$ and $\chi_q$ the Legendre symbol with modulus $q$, for all $n \leq q$ with at most $\ll q^{1 - \varepsilon / 4}$ exceptions. We rule out the case $\kappa' = -1$ using mere multiplicativity. We then rule out the remaining case $\kappa' = 1$ using Siegel's lower bound for $L(1, \chi)$. 

We remark that, unlike Linnik's theorem (where both the exponent and the implied constant 
are effectively computable), our argument
uses Siegel's theorem and therefore the threshold ``sufficiently large in terms of $1 / \varepsilon$" appearing in Theorem~\ref{thm:main} is uncomputable.

\section{Notation \& Acknowledgments}

Throughout $\chi_q(\cdot)$ denotes the Legendre symbol modulo $q$. For $x \in \mathbb{R}$ we define $e(x) := e^{2\pi i x}$. Given a smooth function $f$ we define its Fourier transform as
$$
\widehat{f}(x) := \int_{\mathbb{R}} f(\xi) e(-x \xi) d \xi.
$$
$W$ always denotes a smooth function compactly supported in $[1/2,1]$ with $0 
\leq W \leq 1$ and $\int_{\mathbb{R}} W(x) dx > 0$.
The first author is supported by NSF grants DMS-1802139 and DMS-2301264, a Simons
Collaboration grant and a Simons Fellowship.
The second author acknowledges support of NSF grant DMS-2401106.

\section{Auxiliary lemmas}

\begin{lemma}[Poisson summation for characters {\cite[Exercise~5 \& Lemma~3.1]{IK}}]\label{lem:poisson}
Let $q$ be a positive integer and let $\chi$ be a primitive character modulo $q$.
Let $W \colon \mathbb{R} \to \mathbb{C}$ be a smooth function with compact support.
Then, for any $N > 0$,
\[
\sum_{n \in \mathbb{Z}} \chi(n)\, W\!\biggl(\frac{n}{N}\biggr)
\;=\;
\frac{N \tau(\chi)}{q}
\sum_{m \in \mathbb{Z}} \overline{\chi}(m)\, \widehat{W}\!\biggl(\frac{mN}{q}\biggr),
\]
where $\tau(\chi) = \sum_{a=1}^{q} \chi(a)\, e(a/q)$ is the Gauss sum.
In particular, $|\tau(\chi)| = \sqrt{q}$ for $\chi$ primitive.
\end{lemma}
 
\begin{lemma}[Siegel's theorem {\cite[Theorem~11.14]{MV}}]\label{lem:siegel}
For every $\eps > 0$ and every real primitive character $\chi$ modulo $q$, we have
\[
L(1,\chi) \gg_\eps q^{-\eps}.
\]
\end{lemma}

\begin{lemma}[Large sieve inequality {\cite[Section~7.5]{IK}}]\label{lem:largesieve}
Let $q \in \mathbb{N}$ and $N \geq 1$. For any complex numbers $(a_n)_{n \leq N}$,
\[
\sum_{\chi \pmod{q}} \biggl| \sum_{n \leq N} a_n \chi(n) \biggr|^2
\leq (N + q) \sum_{\substack{n \leq N \\ (n,q) = 1}} |a_n|^2.
\]
\end{lemma}

\begin{lemma} \label{le:hyperbola}
  Let $W : [0, \infty) \to [0,1]$ be smooth and compactly supported in $[1/2, 1]$,
  and let $\chi$ be a primitive character of modulus $q > 1$. Then, for $N \geq q$, 
  $$
    \sum_{n \geq 1} \Big( \sum_{d \mid n} \chi(d) \Big) W\!\Big( \frac{n}{N} \Big)
    = L(1, \chi) N \widehat{W}(0) + O(q^{1/2}).
  $$
\end{lemma}

\begin{proof}
  By Mellin inversion,
  $$
    \sum_{n \geq 1} \Big( \sum_{d \mid n} \chi(d) \Big) W\!\Big( \frac{n}{N} \Big)
    = \frac{1}{2\pi i} \int_{(2)} \zeta(s) L(s, \chi) N^{s}\, \widetilde{W}(s)\, ds,
  $$
  where $\widetilde{W}(s) = \int_0^\infty W(x) x^{s-1} dx$ is entire and satisfies
  $\widetilde{W}(s) \ll_A (1+|s|)^{-A}$ in any vertical strip. We shift the contour to $\Re s = - \delta$ with $\delta = 1 / 100$. Since $\chi$ is non-principal, there is a simple pole at $s = 1$ accounting for the main term. We bound the integral on the line $\Re s = -\delta$ using the functional equation and this yields the error term $\ll N^{-\delta} q^{1/2 + \delta} \ll \sqrt{q}$ for $N \geq q$.  
\end{proof}

\begin{lemma}\label{lem:comb}
Let $A$ and $B$ be two subsets of $(\mathbb{Z}/q\mathbb{Z})^{\times}$
with $|A|+|B|>\varphi(q)$.
Then for every $(n,q)=1$ we can find $a \in A$ and $b \in B$ such that
$n \equiv ab \pmod{q}$.
\end{lemma} 
\begin{proof}
Since $|A| + |B| > \varphi(q)$, for any $n$ with $(n,q)=1$
the sets $A$ and $nB^{-1} := \{n b^{-1} \bmod q : b \in B\}$
satisfy $|A| + |nB^{-1}| > \varphi(q) = |(\mathbb{Z}/q\mathbb{Z})^{\times}|$,
so by the pigeonhole principle they must intersect.
Any element in $A \cap nB^{-1}$ yields the desired representation.
\end{proof}

\section{Proof of Theorem \ref{thm:main}}
 
The proof hinges on the following lemma. 
\begin{lemma}\label{lem:periodicity}
Let $\varepsilon > 0$ and let $q$ be 
sufficiently large  with respect to $1/\varepsilon$.
Let $a$ be an integer coprime to $q$,
let $\kappa \in \{-1, +1\}$, and let $N \geq q^{2 + \varepsilon}$.
If for every $n \equiv a \pmod{q}$ with $n \leq N$ we have
$\lambda(n) = \kappa$, then
\[
\lambda(n + q) = \lambda(n)
\]
for all $n \leq N/q - q$ with $(n, q) = 1$, and
$$
\# \Big \{ 1 \leq n < q : (n,q)=1, \lambda(n) = -1 \Big \} = \frac{\varphi(q)}{2} = \# \Big \{ 1 \leq n < q : (n,q)=1, \lambda(n) = 1 \Big \}.
$$
\end{lemma}

\begin{proof}
Set $\eta := \varepsilon/3$.
For an interval $I = [M, M + q)$ write
\[
A_I = \{n \in I : (n, q) = 1,\; \lambda(n) = -1\},
\qquad
B_I = \{n \in I : (n, q) = 1,\; \lambda(n) = +1\}.
\]
Notice that $|A_I| + |B_I| = \varphi(q)$.

\medskip
\noindent\textbf{Step 1.}
\textit{For every interval $I = [M, M+q) \subset [1, q^{1+\eta})$
we have $|A_I| = |B_I| = \varphi(q)/2$.}

\smallskip
Suppose $|B_I| > \varphi(q)/2$ for some
$I \subset [1, q^{1 + \eta})$. Since $q$ is sufficiently large
with respect to $1/\varepsilon$ and $\eta = \varepsilon / 3$,
we can choose a prime $p \leq q^{\eta}$ coprime to $q$.
Since multiplication by $p$ permutes
$(\mathbb{Z}/q\mathbb{Z})^{\times}$,
the set $pB_I := \{pb \bmod q : b \in B_I\}$
satisfies $|pB_I| = |B_I| > \varphi(q)/2$.

By Lemma~\ref{lem:comb} applied to $(B_I, B_I)$:
for any $(c, q) = 1$ there exist
$b_1, b_2 \in B_I$ with $b_1 b_2 \equiv c \pmod{q}$,
giving $\lambda(b_1 b_2) = +1$ and
$b_1 b_2 \leq q^{2 + 2\eta}$.

By Lemma~\ref{lem:comb} applied to $(B_I, pB_I)$:
there exist $b, b' \in B_I$ with
$pbb' \equiv c \pmod{q}$, giving
$\lambda(pbb') = (-1)(+1)(+1) = -1$ and
$pbb' \leq q^{2 + 3\eta} = q^{2 + \varepsilon} \leq N$.

Taking $c = a$, we obtain integers of both signs in the
class $a$ up to $N$, contradicting the hypothesis.
The case $|A_I| > \varphi(q)/2$ is identical
(swap the roles of $A$ and $B$).
Since $|A_I| \leq \varphi(q)/2$ and $|B_I| \leq \varphi(q)/2$
while $|A_I| + |B_I| = \varphi(q)$,
we conclude $|A_I| = |B_I| = \varphi(q)/2$.

\medskip
\noindent\textbf{Step 2.}
\textit{For every interval $I = [M, M+q) \subset [1, N/q]$
we have $|A_I| = |B_I| = \varphi(q)/2$.}

\smallskip
Suppose $|B_I| > \varphi(q)/2$ for some
$I \subset [1, N/q]$.
Let $J = [1, q)$.
Since $J \subset [1, q^{1 + \eta})$,
Step~1 gives $|A_J| = |B_J| = \varphi(q)/2$.

By Lemma~\ref{lem:comb} applied to $(B_I, A_J)$
(with $|B_I| > \varphi(q)/2$ and $|A_J| = \varphi(q)/2$):
there exist $b \in B_I$, $d \in A_J$
with $bd \equiv a \pmod{q}$.
Then $\lambda(bd) = (+1)(-1) = -1$ and
$bd \leq (N/q) \cdot q = N$.

By Lemma~\ref{lem:comb} applied to $(B_I, B_J)$:
there exist $b_1 \in B_I$, $b_2 \in B_J$ with
$b_1 b_2 \equiv a \pmod{q}$.
Then $\lambda(b_1 b_2) = (+1)(+1) = +1$ and
$b_1 b_2 \leq N$.

Again both signs appear in class $a$ up to $N$,
contradicting the hypothesis.
The case $|A_I| > \varphi(q)/2$ is symmetric.
Since $|A_I| \leq \varphi(q)/2$ and $|B_I| \leq \varphi(q)/2$
while $|A_I| + |B_I| = \varphi(q)$,
we conclude $|A_I| = |B_I| = \varphi(q)/2$.

\medskip
\noindent\textbf{Step 3.}
\textit{For every $n \leq N/q - q$ with $(n,q) = 1$
we have $\lambda(n+q) = \lambda(n)$.}

\smallskip
Let $n \leq N/q - q$ with $(n, q) = 1$.
The intervals
$I = [n, n + q)$ and $I' = [n + 1, n + q + 1)$
both lie in $[1, N/q]$, so Step~2 gives
$|A_I| = |A_{I'}| = \varphi(q)/2$.

Passing from $I$ to $I'$ removes $n$ and adds $n + q$.
Since $q$ is prime and $(n, q) = 1$, we have
$(n + q, q) = 1$, so both $n$ and $n + q$ are counted
by $A$ or $B$.
The cardinality changes as
\[
|A_{I'}| = |A_I|
  - \mathbf{1}[\lambda(n) = -1]
  + \mathbf{1}[\lambda(n + q) = -1].
\]
Since $|A_{I'}| = |A_I|$, we obtain
$\mathbf{1}[\lambda(n) = -1] = \mathbf{1}[\lambda(n+q) = -1]$,
which gives $\lambda(n) = \lambda(n + q)$.
\end{proof}

We will also need the following lemma. 

\begin{lemma} \label{lem:dominant}
Let $\varepsilon > 0$  and let $q$ be a  
large prime. Let $L \geq q^{3/2 + 3\varepsilon}$. Suppose that 
\begin{align} \label{eq:sum_condition}
\sum_{1 \leq n < q} \lambda(n) = 0
\end{align}
and that for all $(n,q) = 1$ with $n \leq L$ we have $\lambda(n + q) = \lambda(n).$
Then there exists a $\kappa' \in \{-1,1\}$ such that 
$$
\lambda(n) = \kappa' \cdot \Big ( \frac{n}{q} \Big )
$$
for all $n \leq q$ with at most $\ll_\varepsilon q^{1 - \varepsilon / 4}$ exceptions. 

\end{lemma}

\begin{proof}
Define the Fourier coefficients
\[
\widehat{\lambda}(\chi) := \sum_{\substack{x = 1 \\ (x,q)=1}}^{q}
\lambda(x)\, \overline{\chi}(x)
\]
for each Dirichlet character $\chi$ modulo $q$.
By condition \eqref{eq:sum_condition}, $\widehat{\lambda}(\chi_0) = 0$.
Since $\lambda$ is $q$-periodic on integers coprime to $q$ up to $L$
we may write
\begin{equation}\label{eq:fourier_expand}
\lambda(n) = \frac{1}{\varphi(q)} \sum_{\chi \neq \chi_0}
\widehat{\lambda}(\chi)\, \chi(n)
\end{equation}
for all $(n,q) = 1$ with $n \leq L$. Let $W$ be a smooth function, compactly supported in $[1/2,1]$, satisfying $0 \leq W \leq 1$ and 
$\int W(x)\, dx > 0$. 
Set $P = q^{1/2 - \varepsilon}$ and $N = q^{1/2 + 2\varepsilon}$.
Since $\lambda$ is completely multiplicative,
$\lambda(n^2 p) = \lambda(n)^2 \lambda(p) = -1$ for every prime~$p$.
Therefore
\begin{equation}\label{eq:key_identity}
 \sum_{\substack{n \in \mathbb{Z} \\ P/2 \leq p \leq P \\ p \text{ prime},\, (np,q)=1}} W \Big (\frac{n}{N} \Big )
\;=\;
\biggl|
\sum_{\substack{n \in \mathbb{Z} \\ P/2 \leq p \leq P \\ p \text{ prime},\, (np,q)=1}}
\lambda(n^2 p) W \Big ( \frac{n}{N} \Big )
\biggr|.
\end{equation}
Note that $N^2 P = q^{3/2 + 3\varepsilon} \leq L$,
so~\eqref{eq:fourier_expand} is valid for $n^2 p$ whenever
$N/2 \leq n \leq N$ and $P/2 \leq p \leq P$. Since $W$ is compactly supported in $[1/2,1]$ the sum over $n \in \mathbb{Z}$ in~\eqref{eq:key_identity} is constrained to $N/2 \leq n \leq N$.
Substituting~\eqref{eq:fourier_expand} and using
$\chi(n^2 p) = \chi^2(n)\chi(p)$, the right-hand side becomes
\begin{equation}\label{eq:char_expand}
\biggl|
\frac{1}{\varphi(q)} \sum_{\chi \neq \chi_0} \widehat{\lambda}(\chi)
\sum_{\substack{(n,q)=1}} \chi^2(n) W \Big ( \frac{n}{N} \Big )
\sum_{\substack{P/2 \leq p \leq P \\ (p,q)=1}} \chi(p)
\biggr|.
\end{equation}
Let $\chi_q$ denote the quadratic character modulo $q$.
Since $\chi_q^2 = \chi_0$, the character $\chi_q$ contributes
\[
\leq \frac{|\widehat{\lambda}(\chi_q)|}{\varphi(q)}
\sum_{\substack{(n,q)=1}} W \Big ( \frac{n}{N} \Big )
\cdot
\biggl| \sum_{\substack{P/2 \leq p \leq P \\ (p,q)=1}} \chi_q(p) \biggr|
\;\leq\;
\frac{|\widehat{\lambda}(\chi_q)|}{\varphi(q)}
\sum_{\substack{n \in \mathbb{Z} \\ P/2 \leq p \leq P \\ (np,q)=1}} W \Big ( \frac{n}{N} \Big ).
\]
We apply H\"older's inequality with exponents $(2,4,4)$ to the sum
over $\chi \neq \chi_0, \chi_q$ in \eqref{eq:char_expand}. Thus it remains to
bound 
$$
\Big ( \frac{1}{\varphi(q)} \sum_{\chi} |\widehat{\lambda}(\chi)|^2 \Big )^{1/2} \cdot 
\Big ( \frac{1}{\varphi(q)} \sum_{\chi \neq \chi_0, \chi_q} \Big | \sum_{n} \chi^2(n) W \Big ( \frac{n}{N} \Big ) \Big |^4 \Big )^{1/4} \cdot \Big ( \frac{1}{\varphi(q)} \sum_{\chi} \Big | \sum_{P/2 \leq p \leq P} \chi(p) \Big |^4 \Big )^{1/4}. 
$$
We now estimate each of these terms. 
By Parseval's identity, 
$$
\frac{1}{\varphi(q)} \sum_{\chi} |\widehat{\lambda}(\chi)|^2
= \varphi(q) \leq q.
$$
For $\chi \neq \chi_0, \chi_q$, the character $\chi^2$ is non-principal
(since $\chi_q$ is the only non-trivial character of order~$2$
when $q$ is prime) and primitive (since $q$ is prime).
Therefore by Poisson summation (Lemma~\ref{lem:poisson}),
\begin{align*}
\sum_{\substack{(n,q)=1}} \chi^2(n) W \Big ( \frac{n}{N} \Big )
& =
\frac{N \tau(\chi^2)}{q}
\sum_{m} \overline{\chi^2}(m)\,
\widehat{W}\!\Bigl(\frac{mN}{q}\Bigr) \\ & = \frac{N \tau(\chi^2)}{q} \sum_{|m| \leq q^{1 + \sigma} / N} \overline{\chi^2}(m) \widehat{W} \Big ( \frac{m N}{q} \Big ) + O_{\sigma, A} (q^{-A})
\end{align*}
for any $A > 0$ and any $0 < \sigma < \varepsilon$, owing to the rapid decay $\widehat{W}(x) \ll_{A} (1 + |x|)^{-A}$ of $\widehat{W}$.
Therefore, 
\begin{align*}
\frac{1}{\varphi(q)} \sum_{\chi \neq \chi_0, \chi_q} \Big | \sum_{n} \chi^2(n) W \Big ( \frac{n}{N} \Big ) \Big |^4  & = \frac{N^4}{q^2 \varphi(q)} \sum_{\chi \neq \chi_0, \chi_q} \Big | \sum_{|m| \leq q^{1 + \sigma} / N} \overline{\chi^2}(m) \widehat{W} \Big ( \frac{m N}{q} \Big ) \Big |^4 + O_\sigma(1) 
\end{align*}
Since $q$ is prime, for each character $\psi$ modulo $q$ there are at most two characters $\chi$ such that $\psi = \chi^2$.   We can therefore bound the above sum by a sum over all the characters $\psi$ modulo $q$, i.e, 
$$
\leq \frac{2 N^4}{q^2 \varphi(q)} \sum_{\psi} \Big | \sum_{|m| \leq q^{1 + \sigma} / N} \psi(m) \widehat{W} \Big ( \frac{m N}{q} \Big ) \Big |^4 + O_\sigma(1). 
$$
By the large sieve, this is 
$$
\ll_\sigma \frac{N^4}{q^3} \cdot \Big ( q + \Big ( \frac{q^{1 + \sigma}}{N} \Big )^2 \Big ) \cdot \Big ( \frac{q^{1 + \sigma}}{N} \Big )^2  \cdot q^{\sigma} \ll N^2 q^{3 \sigma}
$$
since $\sigma \leq \varepsilon$.
Finally, by the large sieve, 
$$
\frac{1}{\varphi(q)} \sum_{\chi} \Big | \sum_{P/2 \leq p \leq P} \chi(p) \Big |^4 \ll \frac{1}{q} \cdot (P^2 + q) P^2 \ll P^2 . 
$$
Combining these together it follows that the contribution of all characters $\chi \neq \chi_0, \chi_q$ to \eqref{eq:char_expand} is 
$$
\ll_{\sigma} q^{1/2} \cdot \sqrt{N}  q^{3 \sigma / 4} \cdot \sqrt{P}   = \sqrt{q N P} \cdot q^{3 \sigma / 4}, 
$$ 
for any $\sigma \in (0, \varepsilon)$. 
Combining everything together, we conclude that
$$
\sum_{\substack{n \in \mathbb{Z} \\ P/2 \leq p \leq P \\ (n p , q) = 1}} W \Big ( \frac{n}{N} \Big ) \leq  \frac{|\widehat{\lambda}(\chi_q)|}{\varphi(q)}
\sum_{\substack{n \in \mathbb{Z} \\ P/2 \leq p \leq P \\ (np,q)=1}} W \Big ( \frac{n}{N} \Big ) + O_\sigma(\sqrt{q N P} q^{3 \sigma / 4})
$$
Since 
$$
\frac{N P}{\log P} \asymp \sum_{\substack{n \in \mathbb{Z} \\ P/2 \leq p \leq P \\ (np, q) = 1}} W \Big ( \frac{n}{N} \Big )
$$ 
and $N P = q^{1 + \varepsilon}$
we conclude, upon taking $\sigma =  \varepsilon / 6 < \varepsilon$, that
$$
1 \leq \frac{|\widehat{\lambda}(\chi_q)|}{\varphi(q)} + O_{\varepsilon}(q^{- \varepsilon / 4}).
$$ 
Since 
$$
\widehat{\lambda}(\chi_q) = \sum_{1 \leq n < q} \lambda(n) \chi_q(n)
$$
and $\lambda(n) \chi_q(n) \in \{-1, 1\}$ for all $1 \leq n < q$,  
it follows that
$
\lambda(n) = \kappa' \chi_q(n)
$
for all $n \leq q$ with at most $\ll_\varepsilon q^{1 - \varepsilon / 4}$ exceptions. 
\end{proof}

In the case $\kappa' = 1$ we can obtain a stronger conclusion. 

\begin{lemma}\label{lem:exceptional}
Let $\delta > 0$ and let $q$ be a prime, sufficiently large
with respect to $1/\delta$.
Let $L \geq q^{3/2}$.
Suppose that
\begin{enumerate}
\item[\textup{(i)}] $\lambda(n) = \chi_q(n)$ for all but at most
$O_\delta(q^{1-\delta})$ integers $n \leq q$ with $(n,q)=1$;
\item[\textup{(ii)}] $\lambda(n+q) = \lambda(n)$ for all
$(n,q)=1$ with $n \leq L$.
\end{enumerate}
Then for all $N \leq q$, 
\[
E(N) \;\ll_\delta\; N q^{-3\delta/4},
\]
where $E(N)$ denotes the number of primes $N / 2 \leq p \leq N$
with $\chi_q(p)=1$.
In particular,
\[
\sum_{\substack{p \leq q \\ \chi_q(p) =1}}
\frac{1}{p} \;\ll_\delta\; q^{-\delta/2}.
\]
\end{lemma}

\begin{proof}
%
Suppose that $q \leq L' \leq L$.
An integer $n \leq L$ with $(n,q) = 1$ has
$\lambda(n) \neq \chi_q(n)$ if and only if its residue
class modulo $q$ is one of the $O(q^{1-\delta})$
exceptional classes from~(i)
(by the periodicity of both $\lambda$ and $\chi_q$ stemming from (ii)),
and each class contributes at most $L'/q + 1 \ll L' / q$
representatives up to $L'$. Thus for any $L'$ with $q \leq L' \leq L$, 
\begin{align} \label{eq:easy_bound} 
\# \{ n \leq L' : \lambda(n) \neq \chi_q(n) \} \ll L' q^{-\delta}.
\end{align}
Define
\begin{align*}
Q^+_N &= \{ N/2<p\le N : p\text{ prime},  \chi_q(p)=1 \}, \\
Q^-_R &= \{ R/2<r\le R : r\text{ prime}, \chi_q(r)=-1 \}.
\end{align*}
Then $E(N) = |Q^+_N|$.
For each $p\in Q^+_N$ and $r\in Q^-_R$, $1=\lambda(pr)\ne \chi_q(pr)=-1$,
so that $pr$ is counted on the left side of \eqref{eq:easy_bound} with $L'=RN$.
As the products $pr$ are distinct,
it follows from \eqref{eq:easy_bound} that
\begin{equation}\label{ENQR}
E(N) |Q^-_R| \ll_\delta RN q^{-\delta} \qquad (q\le NR \le L).
\end{equation}
Firstly, take $R=q$. By (i), $Q^-_q$ contains all but $O_\delta(q^{1-\delta})$ primes in $(q/2,q]$, so $|Q^-_q| \gg \pi(q)\gg q/\log q$ and we conclude that 
\begin{equation}\label{EN1}
E(N)\ll_\delta N q^{-\delta}\log q \ll_\delta N q^{-3\delta/4} \qquad (1\le N\le L/q).
\end{equation}
Next, take $R=L/q$.  By \eqref{EN1}, $Q_R^{-}$ contains all but $O_\delta(R q^{-3\delta/4})$
primes in $(R/2,R]$, and hence $|Q_R^{-}| \gg \pi(R) \gg R/\log R \gg R/\log q$.
Hence, by \eqref{ENQR},
\begin{equation}\label{EN2}
E(N)\ll_\delta N q^{-\delta}\log q \ll_\delta N q^{-3\delta/4} \qquad (q^2/L \le N\le q).
\end{equation}
as $L\ge q^{3/2}$, the ranges in \eqref{EN1} and \eqref{EN2} include all $N\in [1,q]$.
 Finally, let 
$$
E^{\star}(N) := \#\{ p \leq N : \chi_q(p) = 1  \} \leq \sum_{2^k \leq 2 N} E(2^k) \ll_\delta N q^{-3 \delta/ 4}.
$$
Partial summation gives
\[
\sum_{\substack{p \leq q \\ \chi_q(p) = 1}}
\frac{1}{p}
= \frac{E^{\star}(q)}{q}
+ \int_1^q \frac{E^{\star}(t)}{t^2}\, dt \ll_\delta
\frac{1}{q^{3\delta/4}}+\frac{\log q}{q^{3\delta/4}} \ll_\delta q^{-\delta/2}.\qedhere
\]
\end{proof}

We are finally ready to prove the theorem. 

\begin{proof}[Proof of Theorem~\ref{thm:main}]
Let $\varepsilon > 0$ and let $q$ be a prime, sufficiently large
with respect to $1/\varepsilon$.
Since $\varepsilon > 0$ is arbitrary it suffices to prove the theorem
with $q^{5/2+4\varepsilon}$ in place of $q^{5/2+\varepsilon}$.
Moreover we can also assume without loss of generality that $\varepsilon \in (0,1)$. 

Suppose for contradiction that there exists a residue class $a$
coprime to $q$ and a sign $\kappa \in \{-1,+1\}$ such that
$\lambda(n) = \kappa$ for every $n \equiv a \pmod{q}$ with
$n \leq q^{5/2+4\varepsilon}$.
By Lemma~\ref{lem:periodicity} (with $N = q^{5/2+4\varepsilon}$),
\begin{equation}\label{eq:periodicity}
\lambda(n+q) = \lambda(n)
\quad \text{for all } (n,q) = 1 \text{ with }
n \leq q^{3/2+3\varepsilon}.
\end{equation}
By Lemma~\ref{lem:dominant} (with $L = q^{3/2+3\varepsilon}$),
there exists $\kappa' \in \{-1,+1\}$ such that
\begin{equation}\label{eq:correlation}
\lambda(n) = \kappa'\, \chi_q(n)
\quad \text{for all but at most } O_\varepsilon(q^{1-\varepsilon/4})
\text{ integers } n \leq q \text{ with } (n,q)=1,
\end{equation}
where $\chi_q$ is the Legendre symbol with modulus $q$.
The remainder of the proof shows that neither sign of $\kappa'$
is compatible with~\eqref{eq:correlation}.

\medskip
\noindent\textbf{Case $\kappa' = -1$.}
In this case $\lambda(n) = -\chi_q(n)$ for almost all
$n \leq q$ with $(n,q)=1$.
\medskip

We observe that for any two primes $p_1, p_2$ with
$\chi_q(p_1) = \chi_q(p_2)$, the product $p_1 p_2$
satisfies $1=\lambda(p_1 p_2) \neq -\chi_q(p_1 p_2)$.
Let $F$ denote the number of primes $p \leq q^{1/2}$
with $\chi_q(p) = +1$ and let $G$ denote the number
of primes $p \leq q^{1/2}$ with $\chi_q(p) = -1$.
Since $p \neq q$ for all such primes, $\chi_q(p) \in \{-1,+1\}$,
so $F + G = \pi(q^{1/2}) \gg q^{1/2}/\log q$.
The products $p_1 p_2$ with $p_1 < p_2 \leq q^{1/2}$
and $\chi_q(p_1) = \chi_q(p_2) = +1$ are distinct
(by unique factorization), satisfy $p_1 p_2 \leq q$,
and all have $\lambda(p_1 p_2) \neq -\chi_q(p_1 p_2)$
by the observation above.
By~\eqref{eq:correlation}, the number
of integers $n \leq q$ with $(n,q) = 1$ and
$\lambda(n) \neq -\chi_q(n)$ is $O(q^{1-\varepsilon/4})$,
so $\binom{F}{2} \ll q^{1-\varepsilon/4}$.
The same argument with $\chi_q(p_1) = \chi_q(p_2) = -1$
gives $\binom{G}{2} \ll q^{1-\varepsilon/4}$.
Hence $F, G \ll q^{1/2-\varepsilon/8}$, and therefore
$\pi(q^{1/2}) = F + G \ll q^{1/2-\varepsilon/8}$,
contradicting $\pi(q^{1/2}) \gg q^{1/2}/\log q$.

\medskip
\noindent\textbf{Case $\kappa' = +1$.}
In this case $\lambda(n) = \chi_q(n)$ for almost all
$n \leq q$ with $(n,q)=1$.
\medskip

Since $\lambda(p) = -1$ the condition $\lambda(p) = \chi_q(p)$ forces
$\chi_q(p) = -1$ for
almost all primes $p \leq q$.
To exploit this, consider the multiplicative function
$g(n) = \sum_{d \mid n} \chi_q(d) = (1 \ast \chi_q)(n)$.
Since $g(p^k) = 1 + \chi_q(p) + \cdots + \chi_q(p)^k$,
we have $g(p^k) = 0$ whenever $\chi_q(p) = -1$ and $k$ is
odd.
Hence $g(n) = 0$ for every non-square $n$ whose prime
factors all satisfy $\chi_q(p) = -1$. In addition $g(n) \geq 0$ for all $n$.

Every $n < q$ with $g(n) \neq 0$ is therefore either
a perfect square or has a prime factor in
$Q := \{p \leq q : p \text{ prime},\; \chi_q(p) = 1\}$.
The squares contribute
$\sum_{m^2 \leq q} |g(m^2)|
\leq \sum_{m \leq \sqrt{q}} \tau(m^2)
\ll \sqrt{q}\,(\log q)^{O(1)}$.
For the remaining terms, since
$|g(n)| \leq \tau(n)$ and $\tau(pm) \leq 2\tau(m)$,
\[
\sum_{\substack{n < q \\
\exists\, p \in Q:\, p \mid n}} |g(n)|
\;\leq\;
\sum_{p \in Q} \sum_{m < q/p} \tau(pm)
\;\ll\;
q \log q \sum_{p \in Q} \frac{1}{p}
\;\ll_\varepsilon\; q^{1-\varepsilon/9},
\]
where the final bound follows from
Lemma~\ref{lem:exceptional}
(applied with $\delta = \varepsilon/4$,
whose hypotheses are
verified by~\eqref{eq:correlation}
and~\eqref{eq:periodicity}),
which gives $\sum_{p \in Q} 1/p \ll_\varepsilon q^{-\varepsilon/8}$,
and the fact that $q \log q \cdot q^{-\varepsilon/8} \ll q^{1-\varepsilon/9}$
for $q$ sufficiently large.
Hence, for any smooth function $0 \leq W \leq 1$, compactly supported in $[1/2 ,1]$ and with $\widehat{W}(0) > 0$, 
\begin{equation}\label{eq:g_upper}
\sum_{n} g(n) W \Big ( \frac{n}{q} \Big ) \leq \sum_{n < q} g(n) \;\ll_\varepsilon\; q^{1-\varepsilon/9}.
\end{equation}
On the other hand, by Lemma \ref{le:hyperbola},
\[
\sum_{n} g(n) W \Big ( \frac{n}{q} \Big )
= \widehat{W}(0) q\, L(1,\chi_q) + O(\sqrt{q}).
\]
By Siegel's theorem (Lemma~\ref{lem:siegel}),
$L(1,\chi_q) \gg_\varepsilon q^{-\varepsilon/10}$, so
\begin{equation}\label{eq:g_lower}
\sum_{n} g(n) W \Big ( \frac{n}{q} \Big )\;\gg_\varepsilon\; q^{1-\varepsilon/10}.
\end{equation}
The bounds~\eqref{eq:g_upper} and~\eqref{eq:g_lower} are
contradictory for $q$ large enough in terms of
$1/\varepsilon$, since $\varepsilon / 9 > \varepsilon / 10$.

\medskip
Since both cases lead to contradictions, the initial
assumption is false, and Theorem~\ref{thm:main} follows.
\end{proof}

\end{document}